\documentclass{cimart}

\newtheorem{thm}{Theorem}[section] 

\newtheorem{cor}[thm]{Corollary}

\newtheorem{lem}[thm]{Lemma}
\newtheorem{prop}[thm]{Proposition}

\newtheorem{ques}[thm]{Question}

\theoremstyle{definition}

\newtheorem{exmpl}[thm]{Example}

\newcommand\norm{{\operatorname{Norm}}}
\newcommand{\vectwo}[2]{
	\begin{pmatrix}
		#1 \\
		#2
	\end{pmatrix}
}

\title{%
Roots and right factors of polynomials and left eigenvalues of matrices over Cayley--Dickson algebras    
    }

\author{%
    Adam Chapman and Solomon Vishkautsan
    }

\authorinfo[%
    A. Chapman]{
    School of Computer Science, Academic College of Tel-Aviv-Yaffo, Israel}{%
    adam1chapman@yahoo.com
    }

\authorinfo[%
    S. Vishkautsan]{
    Department of Computer Science, Tel-Hai Academic College, Israel}{%
    wishcow@gmail.com
    }

\abstract{%
Over a composition algebra $A$, a polynomial $f(x) \in A[x]$ has a root $\alpha$ if and only $f(x)=g(x)\cdot (x-\alpha)$ for some $g(x) \in A[x]$. We examine whether this is true for general Cayley--Dickson algebras. The conclusion is that it is when $f(x)$ is linear or monic quadratic, but it is false in general. Similar questions about the connections between $f$ and its companion $C_f(x)=f(x)\cdot \overline{f(x)}$ are studied. Finally, we compute the left eigenvalues of $2\times 2$ octonion matrices.
    }



\VOLUME{33}
\YEAR{2025}
\ISSUE{3}
\NUMBERART{1}
\DOI{https://doi.org/10.46298/cm.12613}

\begin{document}

\section{Introduction}

Given an algebra $A$ (in this article an algebra is unital but is not necessarily commutative or associative) with involution $\tau$ over a field $F$, and a choice of $\gamma \in F^\times$, the Cayley--Dickson double $B=A\{\gamma\}$ of $A$ is defined to be $A\oplus A\ell$ with the product 
$$(q+r\ell)(s+t\ell)=qs+\gamma \overline{t}r+(tq+r\overline{s})\ell$$
for any $q,r,s,t \in A$, and where $\overline{x}=\tau(x)$ for any $x\in{A}$. The involution $\tau$ 
extends to $B$ by defining $\overline{q+r\ell}=\overline{q}-r\ell$.
A Cayley--Dickson algebra is obtained by repeating this process several times, starting with a quadratic separable extension $K$ of $F$, with $\tau$ being the unique non-trivial automorphism of order 2 of $K$ acting trivially on $F$. We denote such an algebra $A$ by $K\{\gamma_2,\dots,\gamma_m\}$ where $\gamma_2,\dots,\gamma_m$ are the chosen elements of $F^\times$, in this order. 
When $\operatorname{char}(F)\neq 2$, $K$ is actually $F\{\gamma_1\}$ for the right choice of $\gamma_1 \in F^\times$, and thus $A=F\{\gamma_1,\gamma_2,\dots,\gamma_m\}$.
Such algebras (regardless of the characteristic of the base field) are endowed with a linear trace form $\operatorname{Tr}(q)=q+\overline{q}$ and a quadratic norm form $\norm(q)=q\cdot \overline{q}$ satisfying $q^2-\operatorname{Tr}(q)q+\norm(q)=0$ for any $q \in A$.

When $\dim A=4$, $A$ is called a quaternion $F$-algebra, and when $\dim A=8$ an octonion algebra, and when $\dim A=16$ a sedenion algebra.
Special attention is given in the literature to the Cayley--Dickson algebras constructed over the real numbers by repeatedly choosing $\gamma=-1$, i.e., $\mathbb{R}\{-1,-1,\dots,-1\}$. The first four of those are $\mathbb{C}$, $\mathbb{H}$, $\mathbb{O}$ and $\mathbb{S}$ (resp.\ the complex field, Hamilton's quaternions, the real octonion division algebra and the real sedenion algebra).
These algebras can be described as real Cayley--Dickson algebras with anisotropic norm form, or as locally-complex Cayley--Dickson algebras (a real unital algebra is called locally-complex if any nonscalar element generates a subalgebra isomorphic to $\mathbb{C}$; c.f.\ \cite{brevsar2011locally}). 

The norm form of a Cayley--Dickson algebra $A$ is multiplicative as long as $\dim A \leq 8$. When $\dim A \geq 16$, it stops being multiplicative. For example, when $A=\mathbb{S}$ with its standard basis $e_0,e_1,\dots,e_{15}$ (for the definition of the standard basis in terms of the generators of $\mathbb{S}$ and its multiplication table, see for example \cite[Section 2.2]{CGVZ} and \cite[Section 2]{Cawagas:2004}), the elements $\alpha=e_1+e_{10}$ and $\beta=e_7+e_{12}$ are of norm $2$, so having a multiplicative norm form would mean $\norm(\alpha \beta)=4$, but in fact $\alpha \beta=0$.

In this article we address several questions that came up in the writing of the articles \cite{CGVZ, ChapmanVishkautsan:2022}. In Section~\ref{sec:roots} we discuss the decomposition of the form $f(x)=g(x)\cdot (x-\lambda)$ for roots $\lambda$ of polynomials with coefficients in a Cayley--Dickson algebra. We show that such a decomposition exists for linear and monic quadratic polynomials but does not exist in general for Cayley--Dickson algebras of dimension $\ge 16$.  We also show that in the latter algebras there exist polynomials which have no roots. Finally, we give some counterexamples to questions regarding roots and critical points of Cayley--Dickson polynomials. 

In Section~\ref{sec:eigen} we prove that every octonion matrix has a left eigenvalue, and provide a method for their computation. This provides a direct generalization of Huang and So's results for quaternion matrices in \cite{HuangSo:2001}.

\section{Polynomials over Cayley--Dickson Algebras}\label{sec:roots}

Given a Cayley--Dickson algebra $A$ over $F$, we define the polynomial algebra $A[x]$ to be $A \otimes_F F[x]$; i.e., the indeterminate $x$ behaves like a central element, and in particular, every commutator or associator involving $x$ is trivial.
Every polynomial $f(x) \in A[x]$ can thus be written as $f(x)=c_n x^n+\dots+c_1 x+c_0$ for some $c_0,c_1,\dots,c_n \in A$.
The substitution of $\lambda \in A$ in $f(x)$ is defined by $f(\lambda)=c_n (\lambda^n)+\dots+c_1 \lambda+c_0$. This expression is well-defined as Cayley--Dickson algebras are always power-associative, see \cite{Schafer:1954}.
We say that $\lambda \in A$ is a root of $f(x)\in A[x]$ if $f(\lambda)=0$.

For $f(x)\in A[x]$, the  companion polynomial of $f(x)$ is defined by  $$C_f(x)=\norm(f(x))=f(x)\cdot \overline{f(x)} \in F[x],$$ where $\overline{f(x)}=\overline{c_n} x^n+\dots+\overline{c_1} x+\overline{c_0}$.
It is known that when $\dim A \leq 8$, every root of $f(x)$ is a root of $C_f(x)$, and every root of $C_f(x)$ has an element quadratically equivalent to it, i.e., an element with the same trace and norm, which is a root of $f(x)$; see \cite{Chapman:2020b}.

\subsection{Roots and linear factors of Cayley--Dickson Polynomials}

It is known that when $A$ is a Cayley--Dickson algebra of $\dim A \leq 8$, then $\lambda\in{A}$ is a root of $f(x) \in A[x]$ if and only if $x-\lambda$ is a right factor of $f(x)$, i.e., $f(x)=g(x)\cdot (x-\lambda)$ for some $g(x) \in A[x]$, see \cite{Chapman:2020b}.

\begin{ques}\label{RightFactorQ}
Given a polynomial $f(x)$ over a Cayley--Dickson algebra $A$ and $\lambda \in A$, is it true that $\lambda$ is a root of $f(x)$ if and only if there exists $g(x)\in A[x]$ such that $$f(x)=g(x)\cdot (x-\lambda)?$$
\end{ques}

We begin by settling Question \ref{RightFactorQ} for linear and monic quadratic polynomials:

\begin{prop}\label{linear}
	When $f(x) \in A[x]$ is either linear or monic quadratic, an element $\lambda \in A$ is a root of $f(x)$ if and only if $f(x)=g(x)\cdot (x-\lambda)$ for some $g(x)\in A[x]$.
\end{prop}

\begin{proof}
	When $f(x)$ is linear, $f(x)=ax+b$ for some $a,b \in A$. 
	If $\lambda$ is a root, then \[f(\lambda)=a\lambda+b=0, \] and so $b=-a\lambda$, which means $f(x)=ax-a\lambda=a(x-\lambda)$. In the opposite direction, when $(x-\lambda)$ is a right factor of $f(x)$, there exists $a \in A$ such that $f(x)=a(x-\lambda)=ax-a\lambda$, and then clearly $f(\lambda)=0$.
	
	When $f(x)$ is monic quadratic, $f(x)=x^2+ax+b$. If $\lambda$ is a root, then $f(\lambda)=\lambda^2+a\lambda+b=0$, and so $b=-\lambda^2-a\lambda$, which means $$f(x)=x^2+ax-\lambda^2-a\lambda=(x+\lambda)(x-\lambda)+a(x-\lambda)=(x+\lambda+a)(x-\lambda).$$
	In the opposite direction, if $x-\lambda$ is a right factor, then $f(x)=(x+c)(x-\lambda)$ for some $c \in A$, and so 
	$f(x)=x^2+cx-\lambda x-c\lambda$, which means $f(\lambda)=\lambda^2+c\lambda-\lambda^2-c\lambda=0$.
\end{proof}

We now provide negative answers to Question \ref{RightFactorQ} in general, both for the implication that having a root $\lambda$ means having a right factor $x-\lambda$, and its converse statement.

\begin{exmpl}\label{exmp1}
	Consider $A=\mathbb{S}$ with its standard basis $e_0,\dots,e_{15}$, and denote $\alpha=e_1+e_{10}$ and $\beta=e_7+e_{12}$. The elements $\alpha$ and $\beta$ are zero divisors, and we have $\alpha\beta=\beta\alpha=0$. We remark that  anisotropic Cayley--Dickson algebras over fields of characteristic different from $2$ are reversible, i.e. $\alpha\beta=0$ implies $\beta\alpha=0$; see \cite[Theorem 2.3]{darpo2021}. 
	
	Now let $f(x)=\frac{1}{2}\beta x^2+\beta$. Since $$f(\alpha)=\frac{1}{2}\beta \cdot (\alpha^2)+\beta=\frac{1}{2}\beta\cdot (-2)+\beta=0,$$ $\alpha$ is a root of $f(x)$.
	However, if $x-\alpha$ were a right factor of $f(x)$, there would exist $c\in A$ for which $$f(x)=(\frac{1}{2}\beta x+c)(x-\alpha) = \frac{1}{2}\beta x^2+cx-c\alpha,$$ implying both $c=0$ and $c\alpha=-\beta$, a contradiction.
\end{exmpl}

\begin{exmpl}
	Using the notations from the previous example, 
	consider $f(x)=\beta x^2+\beta x$. This polynomial decomposes as $f(x)=(\beta x+\beta)(x-\alpha)$, but $f(\alpha)=-2\beta$, so $\alpha$ is not a root of $f(x)$.
\end{exmpl}

The question of real Cayley--Dickson algebras being ``algebraically closed" is classic. In this context, ``algebraically closed" means that every polynomial of degree at least $1$ has a root in the algebra. This was proved in the wider setting of general real quaternion polynomials for polynomials with one leading monomial by Eilenberg and Niven (See \cite{eilenberg-niven1944}). This result was subsequently generalized to the real octonions by Jou~\cite{jou1950}, and to more complicated polynomials over any finite-dimensional real composition division algebra in \cite{wilczynski2014}. 
For the special case of left or one-sided polynomials over $\mathbb{O}$ as discussed in this article, we propose the following short proof. 

\begin{prop}\label{Oalgclosed}
	Every polynomial $f(x) \in \mathbb{O}[x]$ of degree at least 1 has a root in $\mathbb{O}$.
\end{prop}

\begin{proof}
	The companion polynomial of $f(x)$ is  defined by  $C_f(x)=\norm(f(x))=f(x)\cdot \overline{f(x)}$. It is a polynomial with \emph{real} coefficients of degree at least $2$, and thus it admits a complex root $\alpha$. Since the complex numbers embed into $\mathbb{O}$, we can assume that $\alpha\in\mathbb{O}$. By \cite[Theorem 3.4]{Chapman:2020a}, the root $\alpha$ of $C_f(x)$ is quadratically equivalent to some root of $f(x)$ in $\mathbb{O}$ (i.e., there exists a root of $f(x)$ with the same trace and norm as $\alpha$).
\end{proof}

\begin{cor} \label{cor:full-decomp}
	Every polynomial $f(x) \in \mathbb{O}[x]$ decomposes as a product of linear factors of the form 
	\begin{equation*}
	f(x) = ((...(c(x - \lambda_n )) \ldots (x-\lambda_3))(x-\lambda_2))(x-\lambda_1).
	\end{equation*}
\end{cor}

\begin{proof}
	This follows from the former proposition and the correspondence between linear right factors and roots of octonion polynomials described at the beginning of the section and proven in \cite{Chapman:2020b}.
\end{proof}

We remark that the $\lambda_i$ for $2\le i \le n$ in Corollary~\ref{cor:full-decomp} are not necessarily roots of $f(x)$.
A natural question to ask is whether this type of decomposition extends to higher dimensional Cayley--Dickson algebras.

\begin{ques}\label{AlgebraicallyClosedQ}
Are all real Cayley--Dickson algebras with anisotropic norm form algebraically closed?
\end{ques}

The answer to Question \ref{AlgebraicallyClosedQ} is evidently negative, as real anisotropic Cayley--Dickson algebras of dimension $\ge 16$ contain zero divisors.
Each such zero divisor $\alpha$ defines a linear map $q\mapsto \alpha q$ on $A$ which is not surjective, implying that there exists some $\beta\in{A}$ such that there is no solution in $A$ to the equation $\alpha x = \beta$. This proves that the linear polynomial $f(x)=\alpha x-\beta$ has no roots in $A$. We now provide a concrete counterexample in $\mathbb{S}$:

\begin{exmpl}\label{Octave}
Let $\alpha$ and $\beta$ be the elements of $A=\mathbb{S}$ from Example~\ref{exmp1}, and let $f(x)=\alpha{x}-\beta$. We show that this linear polynomial has no roots in $\mathbb{S}$. 

Consider the real-valued symmetric bilinear form $\langle a, b \rangle$ associated with the quadratic form $\norm(a)$ on an arbitrary Cayley--Dickson algebra $A$ over $\mathbb{R}$, with $\langle a, a \rangle = \norm(a)$ for all $a \in A$. The bilinear form satisfies the identities $\langle a, bc \rangle = \langle a \bar{c}, b \rangle = \langle \bar{b}a, c \rangle$ for all $a,b,c\in A$ (see \cite[Lemma 6]{Schafer:1954}). Consequently, if $\alpha{x}=\beta$, it follows that $$2=\langle \beta, \beta \rangle = \langle \alpha{x}, \beta \rangle = -\langle {x}, \alpha\beta \rangle = -\langle {x}, 0 \rangle = 0,$$ a contradiction.

\end{exmpl}

\subsection{The Companion Polynomial} \label{sec:companion}
Let $A$ be a Cayley--Dickson algebra over $F$, and $f(x)\in A[x]$. 
Recall that when $\dim A \leq 8$, every root of $f(x)$ is a root of $C_f(x)$, and every root of $C_f(x)$ has an element quadratically equivalent to it. In \cite{CGVZ} it was pointed out that $f(x)=\alpha x \in \mathbb{S}[x]$ has $\beta$ as a root even though $\beta$ is not a root of $C_f(x)=2x^2$ (where $\alpha$ and $\beta$ are as in Example \ref{exmp1}). 
The second part of the property stated in the first line of this paragraph has not been dealt with though, which brings up the following question. 

\begin{ques}
Is every root of $C_f(x)$ quadratically equivalent to some root of $f(x)$, even when $\dim A \geq 16$?
\end{ques}

The answer is in general no:
\begin{exmpl}
In example \ref{Octave} we showed that the polynomial $f(x)=\alpha x-\beta \in \mathbb{S}[x]$ has no roots, but its companion $C_f(x)=2x^2+2$ certainly does (all the elements of trace 0 and norm 1).
\end{exmpl}

\subsection{Isolated and Spherical Roots} \label{sec:isolated}

Another property of $f(x) \in A[x]$ when $\dim A \leq 8$ and its norm form is anisotropic is that every root of $f(x)$ is either isolated (i.e., it is the only element in the quadratic equivalence class that is a root of $f(x)$), or spherical (i.e., every element in the quadratic equivalence class of the root is also a root of $f(x)$).

\begin{ques}
When the norm form of $A$ is anisotropic, is every root of a polynomial $f(x)$ in $A[x]$ either spherical or isolated?
\end{ques}

This is not true when $\dim A \geq 16$:
\begin{exmpl}
The polynomial $f(x)=\alpha x \in \mathbb{S}[x]$ has both $\beta$ and $-\beta$ as quadratically equivalent roots, but $\alpha$ is not a root, even though $\alpha$ is quadratically equivalent to $\beta$.
\end{exmpl}

\subsection{Critical Points}

The Gauss--Lucas theorem states that the critical points of $f(x)\in \mathbb{C}[x]$ (i.e. the roots of the derivative $f'(x)$) are contained in the convex hull of the roots of $f(x)$.
In \cite{GhiloniPerotti:2018} it was shown that this does not extend to $f(x) \in \mathbb{H}[x]$ (here the derivative $f'(x)$ is defined formally).
In \cite[Theorem 4.2]{CGVZ} it was proven that the \emph{spherical} critical points of $f(x)$ (again, $f'(x)$ is defined formally) are contained in the convex hull of the roots of $C_f(x)$ for any $f(x) \in A[x]$ where $A$ is any real anisotropic Cayley--Dickson algebra.

\begin{ques} 
Are the spherical roots of $f'(x)$ contained in the convex hull of the roots of $f(x)$?
\end{ques}
We hereby show that this does not hold:
\begin{exmpl}
The polynomial $f(x)=\frac{1}{3}x^3+x+i \in \mathbb{H}[x]$ has three isolated roots, none of which are spherical. Indeed, the polynomial has three complex roots, but no real quadratic factors, i.e., no two roots are complex-conjugate. If $f(x)$ had a spherical root $\lambda$, it would have had exactly two distinct roots quadratically equivalent to $\lambda$ in $\mathbb{C}$ (see e.g.\ \cite[Theorem 2.2]{zhang1997}; or \cite[Corollary 3.8]{CGVZ} for general locally-complex Cayley--Dickson algebras), and these must be complex-conjugate, leading to a contradiction. Note however, that $f'(x)=x^2+1$ does have spherical roots, most of which are not contained in the convex hull of the roots of $f(x)$, which are all complex.
\end{exmpl}

\section{Roots and Left Eigenvalues}\label{sec:eigen}

When $A$ is a non-commutative algebra, the notion of eigenvalues of a matrix $B$ with coefficients in $A$ divides into two categories: left eigenvalues $\lambda$ which satisfy $B \vec{v}=\lambda \vec{v}$ for some nonzero vector $\vec{v}$, and right eigenvalues which satisfy $B \vec{v}=\vec{v} \lambda$ (in either case we say that $\vec{v}$ is the eigenvector associated to the eigenvalue $\lambda$).
We focus here on left eigenvalues.
The motivating observation is that $\lambda$ is a left eigenvalue of $B \in M_n(A)$ if and only if there exists a nonzero column vector $\vec{v}$ in $A^{n}$ (the direct product of $n$ copies of $A$) for which $(B-\lambda I)\vec{v}=0$, which means $B-\lambda I$ defines (by multiplication from the left) a singular linear endomorphism of $A^{n}$ as an $F$-vector space. 

In the associative case, right eigenvalues are well-understood. In \cite{Lee:1949} it was shown that over $\mathbb{H}$ they can be easily found through the embedding of $M_n(\mathbb{H})$ into $M_{2n}(\mathbb{C})$, and this method was generalized in \cite{ChapmanMachen:2017} to any associative central division algebra.
When $A$ is a quaternion algebra, it was shown in \cite{HuangSo:2001} that the left eigenvalues of $2 \times 2$ matrices over $A$ can be found by solving a quadratic equation. The goal of this section is to do the analogous thing for octonion division algebras.
Note that the special case of $2 \times 2$ and $3 \times 3$ Hermitian matrices over $\mathbb{O}$ was studied in \cite{DrayMongue:2015}.

Let $A$ be an octonion division algebra over a field $F$.
Consider a matrix $B \in M_n(A)$.
Let $\sigma_L(B)$ denote the set of left eigenvalues of $B$, and let 
\[\Sigma_L(B) = \{(\lambda,\vec{v}) : \lambda\in A, \vec{0} \neq \vec{v} \in A^n, B\vec{v}=\lambda \vec{v}\}.\]

\begin{lem} \label{shiftedeigen}
	Let $\lambda$ be a left eigenvalue of a matrix $B\in{M_n(A)}$ and $\vec{v} \in A^n$ an associated eigenvector, where $A$ is an octonion division algebra over a field $F$. Then for any $0\ne{e}\in{A}$, the element $e\lambda$ is an eigenvalue of the matrix $eB$ with an associated eigenvector $\vec{v}e$.
\end{lem}

\begin{proof}
	The algebra $A$ satisfies the Moufang identity $(xy)(zx)=x(yz)x$. Therefore
	for any $e \in A$ we have $(eB)(\vec{v}e)=e(B \vec{v})e = e(\lambda\vec{v})e = (e\lambda)(\vec{v}e)$.
\end{proof}

\begin{lem}\label{triangular}
	Let $B\in{M_n(A)}$ be a lower (upper) triangular matrix with diagonal $(a_1,\ldots,a_n)$, where $A$ is any (not necessarily associative) division algebra over a field $F$. Then $\sigma_L(B) = \{a_1,\ldots,a_n\}$. 
\end{lem}

\begin{proof}
	A direct elementary proof when $A$ is a field does not use associativity or commutativity, so transfers to our case. To summarize, it is clear that $\sigma_L(B) \supseteq \{a_1,\ldots,a_n\}$, since it is easy to see that for each $a_i$ for  $1\le{i}\le{n}$, an associated eigenvector is $\vec{e_i} =  (0,\ldots,1,\ldots 0)^T$ with $1$ in the $i$-th position (Here and in the rest of the article $\vec{v}^T$ denotes the transpose of the vector $\vec{v}$). In the other direction, let $\lambda \in \sigma_L(B)$, and $\vec{v} = (v_1,\ldots,v_n)^T \in A^n$ be an associated eigenvector. The first row of $B$ produces the equation $a_1v_1 = \lambda v_1$ so either $\lambda = a_1$ or $v_1=0$. If $v_1=0$ we continue to the second row to obtain $a_2v_2 = \lambda v_2$ implying either $\lambda = a_2$ or $v_2=0$. Continuing by induction we prove that either $\lambda\in \{a_1,\ldots,a_n\}$ or $\vec{v}=0$, where the latter is a contradiction. 
\end{proof}

For a given $f(x) \in A[x]$,  let $R(f(x))$ be the set of roots of $f(x)$ in $A$.
We denote by $LMR(f(x))$ the union of $R(c\cdot f(x))$ where $c$ ranges over all nonzero elements of $A$.

\begin{lem} \label{specialeigenvector}
	Let $B=\left( \begin{array}{lr}
		a & b\\
		c & d
	\end{array}\right)$ be a matrix defined over an alternative algebra $A$. A vector of the form  $\vec{v} = \vectwo{1}{s} \in A^2$ is an eigenvector associated to an eigenvalue $\lambda \in \sigma_L(B)$ if and only if $s$ is a root of the quadratic polynomial $f(x)=b x^2+(a-d)x-c \in A[x]$ and $\lambda = a+bs$. 
\end{lem}

\begin{proof}
The equation $B\vec{v} = \lambda\vec{v}$ induces the following system of equations:

\begin{equation}
	\left\{
	\begin{array}{l}
		a+bs = \lambda \\
		c+ds = \lambda s \\
	\end{array}
	\right.
\end{equation}

Substituting $\lambda = a+bs$ in the second equation, and using alternative identity $(bs)s = bs^2$ we obtain $f(s)= b s^2+(a-d)s-c = 0$, as required. The other direction follows readily by taking a root $s$ of $f(x)$, setting $\lambda=a+bs$ and comparing $B\vec{v}$ and  $\lambda\vec{v}$.
\end{proof}

\begin{lem} \label{lem:new-moufang}
	Let $M$ be a Moufang loop, and $x,y,z \in M$. Then $x((x^{-1}y)z) = (y(zx))x^{-1}$.
\end{lem}

\begin{proof}
Let $t=x^{-1}y$. Then $y=xt$ and  
\[(y(zx))x^{-1} = ((xt)(zx))x^{-1} = (x(tz)x)x^{-1} = x((tz)(xx^{-1})) = x(tz)=x((x^{-1}y)z).\]
The second equality follows from the Moufang identity $(xt)(zx) = x(tz)x$, and the third by the Moufang identity $(zxz)y = z(x(zy))$ (or the fact that $x$ and $tz$ form a subgroup of $M$).
\end{proof}

\begin{thm}\label{main}
Given a matrix $B=\left( \begin{array}{lr}
a & b\\
c & d
\end{array}\right)$ over a division octonion algebra $A$, if $b=0$ then $\sigma_L(B)=\{a,d\}$, and if $b\neq 0$ then
\[ \Sigma_L(B)=\left\{\left(\lambda,\left(\begin{array}{r} 1 \\ s \end{array} \right) \cdot t\right) : t\in A^\times, s\in R(t^{-1}f(x)), \lambda=a+t((t^{-1} b)s)\right\}, \]
where $f(x)=b x^2+(a-d)x-c.$
 In particular, when $b \in F^\times$, we have
\[ \sigma_L(B)=\{\lambda : b^{-1}(\lambda-a) \in LMR(f(x))\}. \]
\end{thm}

\begin{proof}
When $b=0$, the result follows from Lemma~\ref{triangular}, so we assume that $b\ne{0}$. Now let $\lambda \in \sigma_L(B)$, and $\vec{v} = \vectwo{v_1}{v_2} \in A^2$ be an associated eigenvector. If $v_1=0$ then from $B\vec{v}=\lambda\vec{v}$ we get $b v_2=0$, so $v_2=0$, a contradiction.
Therefore, if $\vec{v} = \vectwo{v_1}{v_2}$ is an eigenvector, then $v_1 \neq 0$.
By Lemma~\ref{shiftedeigen}, we get that $v_1^{-1} \lambda$ is a left eigenvalue of the matrix $v_1^{-1} B$ with associated eigenvector $\vec{v}v_1^{-1} = \vectwo{1}{v_2v_1^{-1}}$. 


By Lemma~\ref{specialeigenvector} we get that $v_2v_1^{-1}$ is a root of $v_1^{-1}f(x)$, where $f(x)=bx^2+(a-d)x-c$, with $v_1^{-1} \lambda = v_1^{-1}a+(v_1^{-1}b)(v_2v_1^{-1})$. Using the Moufang identity $(xy)(zx)=x(yz)x$ on the last equality, we get $v_1^{-1} \lambda = v_1^{-1}a+v_1^{-1}(bv_2)v_1^{-1}$. Multiplying by $v_1$ on the left, we get $\lambda = a+(bv_2)v_1^{-1}$ (Here we used the inverse identity $x^{-1}(xy) = y$ for $x\ne{0}$). Summing up, we know that $v_2v_1^{-1} \in LMR(f(x))$, and so 
\begin{equation} \label{sigma-contained}
 \Sigma_L(B)\subseteq\left\{\left(a+(bv_2)v_1^{-1}, \left(\begin{array}{r} v_1 \\ v_2 \end{array} \right) \right) : v_1\ne 0, v_2v_1^{-1} \in R(v_1^{-1}f(x))\right\}.
\end{equation}
Now let $\left(a+(bv_2)v_1^{-1}, \left(\begin{array}{r} v_1 \\ v_2 \end{array} \right) \right)$ be a pair in the set on the right hand side of \eqref{sigma-contained}. By $v_2v_1^{-1} \in R(v_1^{-1}f(x))$ and Lemma~\ref{specialeigenvector} we get that $\vectwo{1}{v_2v_1^{-1}}$ is an eigenvector of the matrix $v_1^{-1}B$ associated to the eigenvalue $v_1^{-1}a+(v_1^{-1}b)(v_2v_1^{-1})$. Thus by  Lemma~\ref{shiftedeigen} we know $\left(\begin{array}{r} v_1 \\ v_2 \end{array} \right)$ is an eigenvector of the matrix $B$ associated to the eigenvalue $a+(bv_2)v_1^{-1}$ (we used the Moufang identity $(xy)(zx)=x(yz)x$ to simplify the eigenvalue to this form), so  $\left(a+(bv_2)v_1^{-1}, \left(\begin{array}{r} v_1 \\ v_2 \end{array} \right) \right) \in \Sigma_L(B)$. We have thus proved that 
\begin{equation*}
	\Sigma_L(B)=\left\{\left(a+(bv_2)v_1^{-1}, \left(\begin{array}{r} v_1 \\ v_2 \end{array} \right) \right) : v_1\ne 0, v_2v_1^{-1} \in R(v_1^{-1}f(x))\right\}.
\end{equation*}

Setting $t=v_1$ and $s=v_2v_1^{-1}$, we obtain (here we used the identity $x((x^{-1}y)z) = (y(zx))x^{-1}$ from Lemma~\ref{lem:new-moufang})
\[\Sigma_L(B)=\left\{\left(\lambda,\left(\begin{array}{r} 1 \\ s \end{array} \right) \cdot t\right) : t\in A^\times, s\in R(t^{-1}f(x)), \lambda=a+t((t^{-1} b)s)\right\}.\]

When $b \in F^\times$, $\lambda$ depends only on the value of $s$ and not of $t$, and thus
$\sigma_L(B)=\{\lambda : b^{-1}(\lambda-a) \in LMR(f(x))\}$.
\end{proof}

It is known that any matrix with quaternion coefficients has a left eigenvalue, see~\cite{wood1985}. At least for $2\times{2}$ matrices we can say the same for octonion matrices.

\begin{cor}
	Let $B\in{M_2(\mathbb{O)}}$. Then $\sigma_L(B) \ne \emptyset$, i.e.\ the $2\times 2$ matrix $B$ always has a left eigenvalue.
\end{cor}

\begin{proof}
Setting $t=1$ in the theorem, and using the fact that $\mathbb{O}$ is ``algebraically closed" (see Proposition~\ref{Oalgclosed}), there always exists an eigenvector of the form $\vectwo{1}{s}$ for some $s\in\mathbb{O}$ associated to the eigenvalue $a+bs$. 
\end{proof}

\begin{cor}\label{zero}
Suppose $b \neq 0$.
When $d=0$, we have $0 \in \sigma_L(B)$ if and only if $c=0$.
When $d\neq 0$, we have $0 \in \sigma_L(B)$ if and only if there exists a nonzero $t \in A$ satisfying $d^{-1}(ct)-b^{-1}(at)=0$.
\end{cor}

\begin{proof}
 
 Let $\left(\lambda,\left(\begin{array}{r} 1 \\ s \end{array} \right) \cdot t\right) \in \Sigma_L(B)$. By Theorem~\ref{main} we get the two conditions $s\in R(t^{-1}f(x))$ and $\lambda = a+t((t^{-1}b)s)$. 
 
 Setting $ a'=t^{-1}a, b'=t^{-1}b, c'=t^{-1}c, d'=t^{-1}d, 
 \lambda' = t^{-1}\lambda$, the conditions become $b' s^2+(a'-d')s-c' = 0$ and $s=b'^{-1}(\lambda'-a')$. Plugging in $s=b'^{-1}(\lambda'-a')$ into the equation $b' s^2+(a'-d')s-c' = 0$ we obtain a two-sided (non-standard) polynomial equation in $\lambda'$ whose constant term (i.e., the coefficient obtained by setting $\lambda'=0$) is $d'(b'^{-1} a')-c'$.

Assume $d=0$. Then the constant term described above is equal to $-c'$, which is zero if and only if $c=0$. Therefore $0\in \sigma_L(B)$ if and only if $c=0$.

When $d\neq 0$, the constant term is $0$ if and only if  $b'^{-1} a'-d'^{-1} c' = 0$.
Therefore, $0 \in \sigma_L(B)$ if and only if $(b^{-1} t)(t^{-1} a)-(d^{-1} t)(t^{-1} c)=0$.
By multiplying from the right by $t$ and applying the Moufang identity $((xy)z)y=x(yzy)$, the last equation becomes $d^{-1}(ct)-b^{-1}(at)=0$, as required.
\end{proof}

\begin{exmpl}

Consider the matrix $B=\left(\begin{array}{lr}
i & 1\\
ij & j
\end{array}\right)$ over the real octonion division algebra $\mathbb{O}$ with standard generators $i,j,\ell$.
This matrix is defined over the quaternion subalgebra $\mathbb{H}$, and it is invertible as a quaternion matrix, with inverse $B^{-1}=-\frac{1}{2}\left(\begin{array}{lr}
	i & ij\\
	-1 & j
\end{array}\right)$ (as in the complex case, if a quaternion matrix has a one sided inverse, then it has a two-sided inverse; see~\cite[Proposition 4.1]{zhang1997}), and therefore 0 is not in its spectrum.
However, for $t=\ell$ we have $d^{-1}(c t)-b^{-1}(a t)=j^{-1}((ij) \ell)-1^{-1}(i\ell)=0$, so by Corollary \ref{zero} we get $0 \in \sigma_L(B)$.
A suitable eigenvector is $\vec{w}=\left(\begin{array}{lr}
-\ell \\
i\ell
\end{array}\right)$. Note that its right multiple $\vec{w}\cdot \ell^{-1}=\left(\begin{array}{lr}
-1 \\
i
\end{array}\right)$ is not an eigenvector of $B$ anymore, unlike the quaternionic case (where the set of eigenvectors associated to a left eigenvalue is closed under multiplication by a scalar to the right).
Therefore, this matrix, as a matrix in $M_2(\mathbb{O})$ is both invertible and a zero divisor.
Moreover, it is a product of two matrices $B=\left(\begin{array}{lr}
i & 0\\
0 & ij
\end{array}\right)\cdot
\left(\begin{array}{lr}
1 & -i\\
1 & i
\end{array}\right)$
where neither is a zero divisor. In fact, the left matrix only has $i,ij$ in its spectrum by Lemma~\ref{triangular} and the right one does not have 0 in its spectrum by Corollary \ref{zero}.

In order to compute the left eigenvalues of $B$, we shall consider the polynomial $f(x)=x^2+(i-j)x-ij$, and then $\sigma_L(B)=LMR(f(x))+i$.
By \cite[Remark 3.5]{ChapmanVishkautsan:2022}, 
\begin{align*}
LMR(f(x)) &= \{\xi\cdot (i-j)^{-1}(ij+1)+(1-\xi)(ij+1)(i-j)^{-1}+z\ell : 0\leq \xi\leq 1,  \\
&\qquad \norm(z) = \xi\cdot (1-\xi)\cdot \norm([ij-1,(i-j)^{-1}]\} \\
&= \{\xi j+(1-\xi)\cdot (-i)+z\ell : 0\leq \xi \leq 1, \norm(z)=2\xi(1-\xi)\}.
\end{align*}
In particular, for $\xi=0$, we get $-i \in LMR(f(x))$. Thus $0\in LMR(f(x))+i=\sigma_L(B)$.
\end{exmpl}

The arguments presented in the proof of Theorem~\ref{main} can also be applied to an associative division algebra $D$ to describe the left eigenvalues of any $2\times 2$ matrix over $D$. Indeed, since the nonzero elements of an associative division algebra form a Moufang loop with respect to its product,   Theorem~\ref{main} and the preceding lemmas are applicable. Since the algebra is associative, the expression $a+t(t^{-1}b)s$ appearing in Theorem \ref{main} becomes $a+s$, and thus:
\begin{cor}
Given an associative division algebra $D$ and a matrix $B=\left( \begin{array}{lr}
a & b\\
c & d
\end{array}\right)$ over $D$, if $b=0$ then $\sigma_L(B)=\{a,d\}$ and if $b\neq 0$, then 
\[\sigma_L(B)=\{\lambda\in D : b^{-1}(\lambda-a) \in R(bx^2+(a-d)x-c)\}.\] 
\end{cor}
We should mention here that Theorem 2.3 in  \cite{HuangSo:2001} for quaternion matrices is of the form of the corollary, and the proof there works for associative division algebras as well. 

\section*{Acknowledgements}

The authors wish to express their sincere gratitude to the three anonymous referees who read the manuscript carefully and provided many helpful comments and suggestions.


\EditInfo{November 29, 2023}{February 9, 2024}{Ivan Kaygorodov and David Towers}
\end{document}